\documentclass[a4paper,12pt]{amsart}
\usepackage{amssymb}
\usepackage{amscd}

\def\Bbb{\mathbb}
\def\Cal{\mathcal}

\newcommand{\stack}[2]{{\stackrel{\scriptscriptstyle{#1}}{#2}}\phantom{}}
\def\bbR{{\mathbb{R}}}

\newcommand{\el}{B}%{E\hspace{-1mm} L}
\newcommand{\BO}{\stack{o}{B}}
\newcommand{\EO}{\stack{o}{E}}
\newcommand{\GO}{\stack{o}{G}}

\let\i=\iota

\let\w=\omega

\renewcommand{\div}{\operatorname{div}}
\newcommand{\tr}{\operatorname{tr}}
\newcommand{\newc}{\newcommand}

\let\ccdot\cdot
\def\cdot{\hbox to 2.5pt{\hss$\ccdot$\hss}}

\newcommand{\om}{\omega}
\renewcommand{\phi}{\varphi}

\newc{\aI}{\mbox{\boldmath{$ I$}}}
\newc{\aR}{\mbox{\boldmath{$ R$}}}
\newc{\aDeR}{\mbox{\boldmath{$ U$}}_B{}^P{}_C{}^Q}
\newc{\al}{\mbox{\boldmath$ \Delta$}}             
\newc{\nda}{\mbox{\boldmath$ \nabla$}}
\newc{\ad}{\mbox{\boldmath$ d$}}
\newc{\da}{\mbox{\boldmath$ \delta$}}
\newc{\aK}{\mbox{\boldmath{$ K$}}}
\newc{\aL}{\mbox{\boldmath{$ L$}}}

\newtheorem{theorem}{Theorem}[section]
\newtheorem{lemma}[theorem]{Lemma}
\newtheorem{proposition}[theorem]{Proposition}
\newtheorem{corollary}[theorem]{Corollary}

\newcommand{\cC}{{\Cal C}}

\newcommand{\cM}{{\Cal M}}

\newcommand{\cL}{{\Cal L}}

\newcommand{\cS}{{\Cal S}}

\newcommand{\Ric}{\operatorname{Ric}}
\newcommand{\Sc}{\operatorname{Sc}}

\newcommand{\wh}{\widehat}

\let\i=\iota

\newcommand{\nn}[1]{(\ref{#1})}

\newc{\strutdd}{\rule{0mm}{5mm}}

\usepackage{ifthen}

\newc{\tensor}[1]{#1}

\newc{\Mvariable}[1]{\mbox{#1}}

\newc{\down}[1]{{}_{
\ifthenelse{\equal{#1}{;}}{|}{#1}}}

\newc{\up}[1]{{}^{#1}}
\newc{\C}{C}

\newc{\JulyStrut}{\rule{0mm}{6mm}}
\newc{\midtenPan}{\mbox{\sf S}}
\newc{\midten}{\mbox{\sf T}}
\newc{\midtenEi}{\mbox{\sf U}}
\newc{\ATen}{\mbox{\sf E}}
\newc{\BTen}{\mbox{\sf F}}
\newc{\CTen}{\mbox{\sf G}}

\def\sideremark#1{\ifvmode\leavevmode\fi\vadjust{\vbox to0pt{\vss% the remark
 \hbox to 0pt{\hskip\hsize\hskip1em%                          will appear only
 \vbox{\hsize3cm\tiny\raggedright\pretolerance10000%          on the side
 \noindent #1\hfill}\hss}\vbox to8pt{\vfil}\vss}}}%
\input epsf
\input xy
\xyoption{all}

\newtheorem{definition}[theorem]{Definition}

\newtheorem{example}[theorem]{Example}

\newtheorem{remark}[theorem]{Remark}

\newcommand{\bd}{\begin{definition}}
\newcommand{\ed}{\end{definition}}

\newcommand{\br}{\begin{remark}}
\newcommand{\er}{\end{remark}}

\newcommand{\bt}{\begin{tabular}}
\newcommand{\et}{\end{tabular}}

\def\cdot{\hbox to 2.5pt{\hss$\ccdot$\hss}}

\usepackage{amsmath,amsfonts,amssymb}

\begin{document}
\renewcommand{\today}{} \title{Universal principles for Kazdan-Warner
  and Pohozaev-Schoen type identities} \author{A.\ Rod Gover and Bent \O{}rsted}

\address{ARG: (1) Department of Mathematics\\
  The University of Auckland\\
  Private Bag 92019\\
  Auckland 1\\
  New Zealand. \newline
(2) Mathematical Sciences Institute, Australian National University,
\newline\indent ACT 0200, Australia} \email{gover@math.auckland.ac.nz}
\address{
B\O{}: Department of Mathematical Sciences\\
Aarhus University\\
Ny Munkegade\\
DK-8000 Aarhus C\\
Denmark} \email{orsted@imf.au.dk}  

\vspace{10pt}


\renewcommand{\arraystretch}{1}
\maketitle
\renewcommand{\arraystretch}{1.5}

\pagestyle{myheadings}
\markboth{Gover and \O{}rsted}{Kazdan-Warner
  and Pohozaev-Schoen type identities}

\begin{abstract}
  The classical Pohozaev identity constrains potential solutions of
  certain semilinear PDE boundary value problems.  The Kazdan-Warner
  identity is a similar necessary condition important for the
  Nirenberg problem of conformally prescribing scalar curvature on the
  sphere. For dimensions $n\geq 3$ both identities are captured and
  extended by a single identity, due to Schoen in 1988. In each of the
  three cases the identity requires and involves an infinitesimal
  conformal symmetry. For structures with such a conformal vector
  field, we develop a very wide, and essentially complete, extension
  of this picture.  Any conformally variational natural scalar
  invariant is shown to satisfy a Kazdan-Warner type identity, and a
  similar result holds for scalars that are the trace of a locally
  conserved 2-tensor.  Scalars of the latter type are also seen to
  satisfy a Pohozaev-Schoen type identity on manifolds with boundary,
  and there are further extensions. These phenomena are explained and
  unified through the study of total and conformal variational theory,
  and in particular the gauge invariances of the functionals
  concerned.  Our generalisation of the Pohozaev-Schoen identity is
  shown to be a complement to a standard conservation law from physics
  and general relativity.
\end{abstract}

\section{Introduction}
\def\Rm{{\sf R}}
\def\ci{C^\infty}

The current work is concerned with developing an effective and
universal approach to treating and extending three identities,
 each of which plays a central role in the constraint of classes of non-linear
geometric PDE problems.

Of these, the most easily stated arises in the problem of conformally
prescribing scalar curvature; that is of determining, on a fixed
conformal structure $(M^n,c)$, which functions may be the scalar
curvature $\Sc^g$ for some $g\in c$. This problem is especially interesting on
the sphere, where it is known as the Nirenberg problem. While
there are obvious constraints arising from the Gauss-Bonnet theorem,
from the seminal work \cite{KW74} of Kazdan-Warner it follows that there are
positive functions on the 2-sphere $S^2$ that are not the curvatures
of metrics that are pointwise conformal to the standard metric. A
similar result was found in higher dimensions \cite{KW75}, and in
all cases the results are a consequence of an identity satisfied by the first
spherical harmonics. A well-known formulation and extension of these
results is due to Bourguignon and Ezin \cite{BoE}, and is based around
their identity: For any conformal (Killing) vector field $X$, of a
closed Riemannian $n$-manifold $(M,g)$, the scalar curvature satisfies
\begin{equation}\label{BoEr}
\int_M \cL_X \Sc^g~dv_g=0,
\end{equation}
where $\cL_X$ denotes the Lie derivative.

Earlier Pohozaev described an identity which applies to, for example,
star shaped manifolds $M$ with smooth boundary $\partial M$ in
Euclidean space \cite{poh65}.  It was used, for example, to establish
non-existence results for a class of semi-linear variants of
eigenvalue boundary problems. These take the form $\Delta u+\lambda
f(u)=0$ with $u|_{\partial M}=0$. Here $f$ is a non-linear function
that satisfies $f(0)=0$. The Pohozaev identity states
\begin{equation}\label{poh}
\lambda n \int_M F(u) 
+ \frac{2-n}{2} \lambda\int_M f(u)u
=\frac{1}{2}\int_{\partial M} (x\cdot \nu)(\nabla_\nu u)^2;
\end{equation}
$x$ is the Euler vector field, $\nu$ is the outward unit normal,
$\nabla_\nu$ the directional derivative along $\nu$, and $F(u)=\int_0^u f(t)dt$.

Remarkably the identities \nn{BoEr} and \nn{poh} are related. More
precisely there is an identity due to Schoen \cite[Proposition
1.4]{schoen} which, at least for $n\geq 3$, includes both as special
cases: For any conformal vector field $X$ on a Riemannian $n$-manifold
$(M,g)$ with smooth boundary $\partial M$, the following identity
holds
\begin{equation}\label{sch}
  \int_M \cL_X \Sc ~dv_g
  = \frac{2n}{n-2}\int_{\partial M}\left(\Ric-\frac{1}{n}\Sc\cdot g \right)(X,\nu)d\sigma_g ;
\end{equation}
here $\nu$ is the outward normal, and $\Ric$ denotes the Ricci
curvature. This was proved using the Bianchi identities, and 
used as a balancing condition for approximate
solutions to a PDE problem linked to the Yamabe equation.  Since
$\partial M$ may be empty it is clear that \nn{sch} extends
\nn{BoEr} for the cases $n\geq 3$. In Section \ref{pohS} we describe,
for the reader's convenience, how to recover \nn{poh} from
\nn{sch}.

The three identities have had a major impact in non-linear and
geometric analysis, and are still used extensively in the current
literature. This has motivated the development of analogous and
related identities: For Kazdan-Warner type identities recovering or
generalising \nn{BoEr} see for example \cite{AHA}, \cite{Bo86},
\cite{BrO91}, \cite{CY}, \cite{delRob}, \cite{GuoHL}, \cite{V2}; for
the Pohozaev identity \nn{poh} see \cite{PRS}, \cite{PS}, \cite{Wag};
and for Schoen's identity (\ref{sch}) (which is sometimes also
referred to as a ``Pohozaev identity'') \cite{ezin},
\cite{Gursky}. This is by no means a complete list. Many of the works
in the area treat specific curvature quantities, and are motivated by
particular geometric problems. Exceptions include \cite{BrO91} which
gives an analogue of \nn{BoEr} for all heat invariants corresponding
to a conformally covariant operator. Most notably, by an an elegant
and powerful argument, Bourguignon describes in \cite{Bo86} a very
general framework for extending the ``Kazdan-Warner identity''
\nn{BoEr}; this is further developed and applied by Delanoe and Robert
in \cite{delRob}.

The identities and works mentioned suggest the following problems: For
what scalar invariants $V=V(g)$ (replacing/generalising $\Sc^g$) do we
expect an analogue of the classical Kazdan-Warner identity \nn{BoEr}?
Any such identity gives an immediate constraint for conformal
curvature prescription on the sphere. Similarly, for what scalar
invariants $V=V(g)$ do we expect an analogue of \nn{sch}? Note that
this identity gives a non-trivial constraint in a vastly wider range
of geometric structures, so any extension has great potential for
application.  The third main problem is to precisely relate the two
types of identity.  For example if, in some general situation of
closed manifolds, $V$ satisfies an analogue of the Kazdan-Warner
identity then do we expect it to also satisfy the Schoen identity
\nn{sch} on manifolds with boundary? That there should be some
subtlety here is clear from the factor of $\frac{1}{n-2}$ in \nn{sch};
Schoen's construction apparently does not recover \nn{BoEr} in
dimension 2.

In the current work we obtain essentially complete answers to the
questions posed by showing that a closely related set of general
principles underlie the Kazdan-Warner and Pohozaev-Schoen type
identities. (Here we restrict to the case where $X$ is a conformal
vector field. There are clearly extensions to related settings, but
this will be taken up elsewhere.)  The principles involved are
strongly related to the notion of symmetry and conservation that dates
back to the work of D.\ Hilbert and E.\ Noether, and indeed this is
our starting point in Section \ref{HN}.  Overall we obtain very
general extensions of the Kazdan-Warner and Pohozaev-Schoen
identities. Concerning the former, the main results are Theorems
\ref{ukw}, Corollary \ref{aa}, Theorem \ref{nocvt}, and Theorem
\ref{bway}. The first three of these show that an identity of the type
\nn{BoEr} is available for any natural scalar invariant which is
conformally variational (as defined in Section \ref{cvar}) for
suitable functionals of increasing generality; in each case the result
is a direct consequence of symmetry invariance, or in other words of a
gauge invariance, in the action functional concerned. The last Theorem
\ref{bway} extends these results to show that in fact {\em any}
conformally variational natural scalar satisfies such an identity.  In
this case the argument (cf.\ \cite{Bo86,delRob}) is less direct and uses now
the invariance of a 1-form on the space of metrics (in a conformal
class), combined with the Lelong-Ferrand-Obata theory \cite{L-F,O}.

The last mentioned approach appears to be necessary for a class of
critical cases, but it misses the connection with the Schoen type
identities \nn{sch}. On the other hand the very simple argument behind
Theorem \ref{ukw} involves specialising total metric variations, and
so is linked to locally conserved 2-tensors (as explained in Section
\ref{cvar}). Through this its proof is intimately connected to Theorem
\ref{maini} which extends \nn{sch} to an identity that holds for the
trace and trace-free parts of any locally conserved 2-tensor. This is
a very large class of invariants that need not be natural (see e.g.\
Section \ref{se}).  It is precisely the difference between Theorem
\ref{ukw} (or Theorem \ref{maini}) and Theorem \ref{bway} that is
behind the $\frac{1}{n-2}$ factor mentioned earlier, and the
generalisation of this phenomenon.  See also Corollary \ref{ci} and
the discussion below.

In Section \ref{consS} we show that the
generalised Schoen identity of Theorem \ref{maini} is a precise
complement to the usual conservation theory extant in the Physics
literature. 

The main results mentioned, and their proofs, appear to unify,
simplify, and considerably extend most of the existing related results
in the literature; see Section \ref{exs} where we show a number of new
results, as well the simplification and unification of a number of
recent particular results in the literature. Specific examples treated
include Gauss-Bonnet curvatures, Q-curvatures, renormalised volume
coefficients, and the mean curvature of a conformal immersion.
Although we do not directly discuss extensions of the Pohozaev
identity \ref{poh} it is clear that such can be obtained from Theorem
\ref{maini} by, for example, an analogue of treatment in Section
\ref{pohS}.

ARG would like to thank Alice Chang, Paul Yang, Matt Gursky and
Fr\'{e}d\'{e}ric Rochon for useful discussions at the meeting
``Geometric and Nonlinear Partial Differential Equations'' held at
Mission Beach Resort, Queensland, 2010. A draft of this work was
presented there. Discussions with Robin Graham and Andreas Juhl at the
Tambara Institute of Mathematics worskop `` Parabolic Geometries and
Related Topics I'', November 2010, are also much appreciated. In particular
Graham provided an answer to a question posed in the first draft, see
Theorem \ref{HET}.

\section{The Hilbert-Noether identities for gradients} \label{HN}

Until further notice we shall suppose we work on a closed (compact
without boundary) oriented connected manifold $M$, of dimension $n\geq
2$ and usually equipped with a Riemannian metric $g$. However we also
consider the space $\cM$ of such metrics on $M$, that space equipped
with the compact open $C^\infty$ topology. For simplicity all structures and 
sections throughout shall be considered smooth ($C^\infty$).

A real valued functional $\cS$ on $\cM$ is called a {\em Riemannian
  functional} if it is diffeomorphism invariant in the sense that it
satisfies
\begin{equation}\label{Rinv}
\cS(\phi^* g)= \cS(g),
\end{equation}
for all $g\in \cM$, and for all diffeomorphisms $\phi:M\to M$.

A {\em natural scalar (Riemannian) invariant} (see e.g.\
\cite{Stredder}) is a scalar valued function which is given by a
universal expression, which is polynomial in the finite jets of the
metric and its inverse, and which has the property that for any
diffeomorphism $\phi:M\to M$ we have
\begin{equation}\label{natural}
\phi^* L(g)= L(\phi^* g).
\end{equation}
An important class of Riemannian functionals, and our main (though
certainly not exclusive) focus here, arise from the integral of such
Lagrangians: that is $g\mapsto \cS (g)$ where
\begin{equation}\label{action}
\cS(g)= \int_M L(g) dv_g
\end{equation}
where $dv_g$ is the metric measure.

\begin{remark}\label{natc} One may construct natural invariants in an 
  obvious way by complete contractions, using the metric, its inverse,
  and the volume form, of expressions polynomial in the Riemann
  curvature, and its Levi-Civita covariant derivatives. In fact all
  natural invariants arise this way as follows by a well known
  argument using Weyl's classical invariant theory and Riemann normal
  coordinates, see e.g.\ \cite{ABP}.
\end{remark}

\subsection{Total metric variations}\label{total} 
The tangent space to $\cM$ is naturally identified with the (smooth)
section space of $S^2M$. A differentiable Riemannian functional is
said to have a {\em gradient} $B(g)$ at $g$, if $B(g)$ is a smooth
section of $S^2M$ and, for all $h\in \Gamma (S^2M)$,
\begin{equation}\label{core}
\cS'(g)(h)=\int_M (h,B(g)) dv_g
\end{equation}
where $(\cdot,\cdot)$ denotes the local pairing of tensors given by
metric contraction. In an abstract index notation we shall write
$B_{ab}$ for $B(g)$.

If we specialise now to $h$ arising from the
pullback along a diffeomorphism generated by a
vector field $X$, then $h=\cL_X g$ and we have
\begin{equation}\label{main}
0=\int_M (\cL_X g ,\el (g))  dv_g,
\end{equation}
from the diffeomorphism invariance of the Riemannian functional
$\cS$.
In terms of the Levi-Civita connection $\nabla$ (for $g$), we have $(\cL_X
g)_{ab}= 2\nabla_{(a}X_{b)}$. 
Thus integrating by parts in \nn{main}
we see that
$$
0=\int_M X^b\nabla^a\el_{ab}  dv_g.
$$
Since the $X^b$ is arbitrary we conclude 
\begin{equation}\label{div}
\nabla^a\el_{ab} =0,
\end{equation}
and we shall say that $B$ is {\em locally conserved}.
This is a standard identity for the gradient of a Riemannian
functional, and is attributed to Hilbert \cite{Besse,Hil}.
Identities derived from symmetries or ``gauge invariance'', such as
this, are often called Noether identities in the literature. 

Now we consider the case where $\cS(g)$ is given by a natural
Lagrangian, as in \nn{action}.  It follows from the result mentioned
in Remark \ref{natc}, and integration by parts, that each directional
derivative of $\cS(g)$ is of the form \nn{core} where $B(g)$ is a
natural (tensor-valued) invariant. From this in turn we conclude that
$\cS(g)$ is differentiable and so the above discussion applies
immediately; in particular the natural tensor $B=B(g)$ satisfies
\nn{div}. 

\begin{remark} \label{others} Although the argument above has used a
  compact Riemannian setting, as an aside here we note the following:
  since $B$ is given by a universal expression in terms of the
  Riemannian curvature and its covariant derivatives, it follows that
  the local result \nn{div} holds on any manifold and in any
  signature.
\end{remark}

To see how non-trivial results may arise from the diffeomorphism
invariance of an action it is useful to understand, via an
infinitesimal argument, how the gradient is generated in the case of a
natural Lagrangian function $L=L(g)$.  Consider a curve of metrics
$g^t$ through $g=g^0$.  Calculating the derivative of \nn{action} at
$t=0$ involves computing the linearisation of $L(g)$ (at $g$),
$$
L'(h):=\dfrac{d}{dt}\Big|_{t=0}L (g^t),
$$
and also the contribution from the measure:
$$
\dfrac{d}{dt}\Big|_{t=0} dv_{g^t}= \frac{1}{2}g^{ab}h_{ab} dv_g .
$$
Putting these together we have 
\begin{equation}\label{precore}
  \dfrac{d}{dt}\Big|_{t=0}\cS(g^t)= 
\int_M \big(L'(h)+ \frac{1}{2} L(g) g^{ab}h_{ab} \big)dv_g .
\end{equation}
However for $h$ arising from an infinitesimal diffeomorphism
we have, as mentioned, $h=\cL_X g$. Thus
$\frac{1}{2}g^{ab}h_{ab}=\nabla_a X^a =\div X$.  On the other hand the
infinitesimal version of the naturality condition \nn{natural} is
\begin{equation}\label{inf}
L'( \cL_X g)= \cL_X L (g)
\end{equation}
and so for $h=\cL_X g$ we have $(L'(h)+ \frac{1}{2} L(g)
g^{ab}h_{ab})= \div (L(g) X)$  whence the right hand side of
\nn{precore} is zero.  The non-trivial identity \nn{core} arises by
calculating in another order. We first integrate \nn{precore} by parts
to yield \nn{core}, and then proceed as argued earlier. So the
information contained in the difference between the two ways of
calculating arises entirely from \nn{inf}.

\subsection{Generalised energy-momentum tensors}\label{se}

The local conservation of natural gradients is a unifying feature
in the discussion which follows.  In fact, as
we shall see, a broader class of gradients satisfy \nn{div}.
 Suppose that rather than restrict to $L$ being a natural
scalar invariant of $(M,g)$, we allow $L$ as follows.  We assume $L$
is a scalar valued function which is given by a universal expression,
which is polynomial in the finite jets of the metric and its inverse,
and also in the finite jets of a collection of other fields that we
shall collectively denote $\Psi$ (and regard as a single field). So we
may write $L=L(g,\Psi)$.  The fields that make up $\Psi$ may be tensor
fields, but also could include for example connections. We shall
not be concerned with the details; it is rather naturality in this
context that is important.  We shall insist that $L$ satisfies
\begin{equation}\label{natural2}
\phi^* L(g, \Psi)= L(\phi^* g, \phi^*\Psi)
\end{equation}
for any diffeomorphism $\phi:M\to M$. So certainly we require that the
nature of the fields $\Psi $ is such that their pullback under
diffeomorphism makes sense, but this is a very weak
restriction. 
We shall call such $L(g,\Psi)$ {\em coupled scalar invariants}.

Now we assume that 
$$
\cS(g,\Psi):=\int_M L (g,\Psi) dv_g
$$
is separately Frechet differentiable with respect to $g$ and $\Psi$,
and that there are respective partial gradients $B(g,\Psi)$, $E(g,\Psi)$,
satisfying
$$
(D_1\cS(g,\Psi))(h)=\int_M (h,B(g,\Psi)) dv_g
$$
and 
$$
(D_2\cS(g,\Psi))(h)=\int_M \langle \psi,E(g,\Psi) \rangle dv_g
$$
where $\psi$ is in the formal tangent space at $\Psi$ to the field
(system) $\Psi$ and $\langle \cdot , \cdot \rangle$ is the pointwise
dual pairing that arises naturally in the problem. (In the other display the notation is as in \nn{core}.)
We shall refer to
$B(g,\Psi)$ as the metric gradient.

The equation $E(g,\Psi)=0$ is a generalised Euler-Lagrange system.  We
have the following result.
\begin{theorem}\label{gg} On $(M,g)$,  let $\Psi_0$ be a solution of the
  generalised Euler-Lagrange system
$$
E(g,\Psi)=0.
$$
Then the metric gradient $B(g,\Psi_0)$ is locally conserved, that is 
\begin{equation}\label{div2}
\nabla^a B_{ab}(g,\Psi_0)=0.
\end{equation}
\end{theorem}
\begin{proof} From \nn{natural2} it follows that $S(g,\Psi)$ is
  diffeomorphism invariant. Thus, differentiating $S(g,\Psi)$ along
  the pullback of an infinitesimal diffeomorphism generated by a
  vector field $X$, and using the chain and product rule under the
  integral, we have a generalisation of \nn{main}, viz.\
$$
0=\int_M \Big( (\cL_X g , B (g,\Psi)) + \langle \cL_X \Psi , E(g,\Psi)
\rangle \Big) dv_g ,
$$
where the derivative of \nn{natural2} is used. 
Thus if we calculate along $\Psi_0$ satisfying $E(g,\Psi_0) =0$, then
this reduces to $0=\int_M (\cL_X g , B (g,\Psi)) dv_g$ and we
argue as below \nn{main} to conclude \nn{div2}.
\end{proof}

\begin{remark} The argument above is a minor variant of that in \cite{HE},
  which treats the case that $L$ depends on at most first covariant
  derivatives of $\Psi$. In that setting $E(g,\Psi_0)=0$ gives the
  standard Euler-Lagrange equations of continuum mechanics and they
  term $B(g,\Psi_0)$ an ``energy-momentum tensor''. In certain
  contexts the same $B(g,\Psi_0)$ is sometimes termed a stress-energy
  tensor \cite{B,BE}.
\end{remark}

\subsection{Conformal variations}\label{cvar}

On a manifold $M$, a natural scalar invariant $V$ is said to be 
  {\em conformally variational} within a conformal class of metrics
$\cC=\{\wh{g}=e^{2\Upsilon}g\mid\Upsilon\in C^\infty(M)\}$ if there is a
functional $\cS(g)$ on $\cC$ with 
\begin{equation}\label{Fbul}
\cS^\bullet(g)(\w)=2 \int_M\w  V \,dv_g\,,\qquad\mbox{all }\w\in\ci(M).
\end{equation}
As above $dv_g$ is the Riemannian measure, and here 
\begin{equation}\label{bul}
\cS^\bullet(g)(\w):=\dfrac{d}{d s}\Big|_{s=0}\cS(e^{2 s \w}g).
\end{equation}
In \nn{bul}, the curve of metrics $e^{2 s \w}g$ may be replaced by any
curve with the same initial tangent $g^\bullet=2\w g$.  The property
of being variational can depend both on $L$, and on the conformal class
$\cC$.  

We shall consider first two important cases, with the first case  as follows. 
\begin{definition}\label{cv}
  We shall say that $V$, a natural scalar invariant, is
  \underline{naturally} conformally variational if it arises as in
  \nn{Fbul} above from a Riemannian functional $\cS$ that admits a
  gradient (as in \nn{core}) for any $g\in \cC$.
\end{definition}

Suppose now  $\cS$ is as in Definition \ref{cv} and we calculate
\nn{bul} via a specialisation of the total metric variation
computation \nn{core}.  It follows that
\begin{equation}\label{confv}
\cS^\bullet(g)(\w)=2\int_M (\om g, \el) \,dv_g
\end{equation}
whence, in particular, $V=g^{ab}\el_{ab}$. We summarise this observation.
\begin{lemma}\label{Vori} If $\cS$ is a Riemannian functional with 
  gradient $B_{ab}$ at $g$, then the function $V$
  in \nn{Fbul} is given by $g^{ab}B_{ab}$.
\end{lemma}

Recall that for total metric variations the key integral relation
underlying the Hilbert-Noether identity is \nn{main}.   Comparing this with
\nn{confv} we see that, in the restricted setting of conformal
variations, \nn{main} still yields constraints provided $\cL_Xg=2\om
g$. But this exactly means that $X$ is a conformal vector field and
$\om=\frac{1}{n}\div X$. 
Then \nn{main} states
$$
0= \int_M (\div X)V dv_g.
$$
So, integrating by parts, we have the following.
\begin{theorem}\label{ukw}
  If $V$ is naturally conformally variational, then for any conformal vector
  field $X^a$ on a closed Riemannian manifold $(M,g)$, we have
\begin{equation}\label{confN}
0= \int_M (\div X) V dv_g= - \int_M (\cL_X V) dv_g.
\end{equation}
\end{theorem}

One might suppose that Definition \ref{cv}, as used in Theorem
\ref{ukw}, is restrictive. In fact in most cases it is not. To make
this precise we need a further definition.  A natural invariant $L$
(possibly tensor valued) is said to have {\em weight} $\ell$ if
uniform dilation of the metric has the effect $L[A^2g]=A^{\ell}L[g]$
for all $0<A\in\bbR$.  For example, the scalar curvature has weight
$-2$. It is not essentially restrictive to consider only invariants of a
well defined weight, since it is easily shown that any natural scalar
invariant is a sum of such. The key to the claim that began this
paragraph is the following result.
\begin{proposition} \cite{BrGovar} \label{self} If~ $V$, of weight
  $\ell\neq -n$, is a conformally variational local scalar invariant on a
  closed Riemannian conformal $n$-manifold $(M,\cC)$, then
\begin{equation}\label{selfe}
\cS(g):=(n+\ell)^{-1}\int_M V dv_g
\end{equation}
is a Riemannian functional for $V$ in $\cC$; that is  \nn{Fbul} holds.
\end{proposition}
Now by the discussion of natural Lagrangians in Section \ref{total}, it
follows that \nn{core} holds for $\cS$ as in \nn{selfe}, and so $\cS$
satisfies Definition \ref{cv}. Thus we have the following.
\begin{corollary}\label{aa} On a closed Riemannian $n$-manifold a 
natural scalar invariant ~$V$, of weight $\ell\neq -n$,
  is conformally variational if and only if it is naturally
  conformally variational.
\end{corollary}

The scalar curvature is well known to be conformally variational and
so Theorem \ref{ukw} certainly extends the results of Bourguignon-Ezin
\cite{BoE} for the scalar curvature in dimensions $n\geq 3$. In fact
conformally variational invariants are not at all rare, and so the
extension is vast; we shall take up this point in Section
\ref{exs}.

Next we show that a slight variant of the above also recovers and extends
the identity from \cite{BoE,KW74} for the Gauss curvature in
dimension 2. Above we used that it is insightful to use the gradient
$B$ when this is available. That observation will also be critical in
the next section. However the existence of a total metric variation
gradient, as in \nn{core}, is not necessary to see a Kazdan-Warner
type identity arise from gauge invariance.
\begin{definition}\label{nocv}
  We shall say that $V$, a natural scalar invariant, is
  \underline{normally} conformally variational if it arises via \nn{Fbul} 
with $\cS$ a Riemannian functional.
\end{definition}
Note that this is a strictly broader class of invariants than above:
if $V$ is naturally conformally variational then it is normally
conformally variational.
\begin{theorem}\label{nocvt}
  The identity \nn{confN} holds if we assume only that $V$ is normally
  conformally variational (with also the other conditions of Theorem
  \ref{ukw} imposed).
\end{theorem}
\begin{proof}
  We follow the idea of Section \ref{total}, but restrict at the
  outset to the case that $X$ is a conformal vector field. Again from
  the diffeomorphism invariance of the Riemannian functional we have
  $0=\cS'(g)(h)$ where $h=\cL_X g$. But $h$ is a conformal variation:
  $h=\frac{2}{n}(\div X) g$. So $\cS'(g)(h)=\cS^\bullet(g)(
  \frac{1}{n}\div X)$ and since, by assumption, $V$ and $\cS$ are related
  by \nn{Fbul} the result follows.
 \end{proof}

 \begin{example}\label{poly} On a closed Riemannian 2-manifold 
   $(M,g)$ if we take $\cS(g):=\det \Delta_{g}/ A(g)$, where $A(g)$ is
   the total area and $\det \Delta_{g}$ is the functional determinant
   of the Laplace-Beltrami operator, then there is the Polyakov
   formula \cite{OPS,P} for conformal variation
$$
\cS^\bullet(g)(\om)=c\cdot \int_M \om Q dv_g
$$
where $Q$ is the Gauss curvature and $c\neq 0$ is a constant. $\cS(g)$
is a Riemannian functional and so we conclude from Theorem \ref{nocvt}
that for any conformal vector field $X$ on $M$ we have $\int_M \cL_X
Q~dv_g =0$.
\end{example}

\begin{remark}
  In view of the derivations in Theorems \ref{ukw} and \ref{nocvt} it
  is clear that the Kazdan-Warner identities are related to
  Noether-Hilbert principles. Note here we do not expect an analogue of
  \nn{div}: The result here is necessarily global, since the common
  ground between \nn{main} and \nn{confv} involves conformal vector
  fields which are global objects.
\end{remark}

From the proof of Theorem \ref{nocvt} it is evident that we may obtain
an identity at a particular $g_1\in \cC$ without the full force of
\nn{Rinv}. Indeed we simply need $\cS'(g_1)(h)=0$ where $h$ is $\cL_X
g_1$ and $X$ a conformal vector field. If $V(g)$ is a conformally
variational natural invariant this is achieved by the functional
$\cS(g)=\int_M \om V(g) dv_g$ on $\C$, where $g=e^{2\om}g_1$, $\om \in
C^\infty (M)$. This follows from the following argument, which is a trivial 
adaption of a result from \cite{Bo86,delRob}.
\begin{theorem}\label{bway} Suppose that $X$ is a conformal vector 
  field on a closed Riemannian conformal manifold $(M,\cC)$, and that
  $V=V(g)$ is a conformally variational natural scalar invariant. Then
$$
\int_M (\div X) V(g) ~dv_g .
$$
is independent of the choice of metric $g\in \cC$, and hence is zero. 
\end{theorem}
\begin{proof}
  Fix any metric $g_0\in \cC$. If $V$ is conformally variational then
  the linearisation of the map $\om \mapsto V(e^{2\om} g_0)$, $\om\in
  C^\infty(M)$, is formally-self-adjoint (see e.g.\ \cite{BrGovar}).
Identifying
  $C^\infty (M)$ with the tangent space to $\cC$, it follows that the
  1-form on $\cC$
$$
C^\infty (M)\ni \om \mapsto \int_M \om V(g)~ dv_g
$$
is closed \cite{Bo86,BrGovar}. Now suppose that $\tilde{X}$ is the vector
field on $\cC$ induced by a conformal diffeomorphism $X$ on $M$. From
the diffeomorphism invariance of this 1-form it is annihilated by
$\cL_{\tilde{X}}$. Then using the Cartan formula $\cL_{\tilde{X}}=
d\i_{\tilde{X}} +\i_{\tilde{X}} d$, and the identification of
$\tilde{X}$ with $\frac{1}{n}\div X$, it follows that $ \int_M (\div
X)V(g) ~dv_g $ is constant on $\cC$ as claimed.

It follows that if there is a metric $g_0\in \cC$ such that $V(g_0)$
is constant then for any metric $g\in \cC$ we have
$$
\int_M \cL_X V ~dv_g=0.
$$
In particular this holds on the sphere $S^n$ with its standard
conformal structure. However by the Lelong-Ferrand-Obata theorem \cite{L-F,O}
if $M$ is any other conformal manifold then $X$ is necessarily a
Killing vector field. In that case we have $\cL_X V =\div (V X)$ and so 
$\int_M \cL_X V ~dv_g=0$.
\end{proof}

\begin{remark}\label{215} 
  While this Theorem gives the strongest result, it uses a less direct
  argument than that of Theorems \ref{ukw} and \ref{nocvt}, and this
  argument partly loses contact with the Hilbert-Noether principles,
  and in most cases is not necessary (as follows from Corollary
  \ref{aa}). Most importantly, as we shall see below, the proof of
  Theorem \ref{ukw} naturally suggests, and links it to, a generalisation of the
  Schoen identity.

  On the other hand for natural invariants of weight $-n$ we expect to
  need stronger arguments: for example if $V$ is a conformal covariant
  of weight $-n$, then $V dv_g $ is a conformally invariant $n$-form
  and so $\int_M V dv_g$ is conformally invariant. It is easily seen
  that such a matching of weights between $dv_g$ and $V$ causes a
  breakdown in the argument of Theorem \ref{ukw}.

  The result in \cite{delRob} corresponding to Theorem \ref{bway} uses
  that the linearisation of $\om \mapsto V( e^{2\om}g)$ is formally
  self adjoint, without any explicit mention that $V$ is
  variational. But for a natural scalar invariant this
  self-adjointness condition is equivalent to it being conformally
  variational, as follows from a trivial variant of \cite[Lemma
  2(ii)]{BrGovar}.
\end{remark}

\medskip

\section{Manifolds with boundary and conservation} \label{bc}

Let $M$ be Riemannian manifold with boundary $\partial M$. To avoid
unnecessary restriction we allow here the possibility that $\partial
M$ is the empty set.  In this setting, and using a different approach
to the above, we derive a result that strictly generalises 
Theorem \ref{ukw} and the Schoen identity \nn{sch}. 

\subsection{A generalisation of the Schoen identity}\label{subsch}
On $M$, let $B$ be a symmetric 2-tensor with compact support, 
and $X$ any tangent vector field. 
Then by the Gauss formula for Stokes' Theorem,
$$
\int_M  \nabla^a(B_{ab}X^b) dv_g=\int_{\partial M} B_{ab} X^a \nu^b d\sigma_g,
$$
where $\nu$ and $d\sigma_g$ are, respectively, the outward unit normal
and the induced metric measure along $\partial M$.

Now if the tensor $B$ is locally conserved, meaning that $ \nabla^a
B_{ab}=0 $, then
\begin{equation}\label{key1}
2 \nabla^a(B_{ab}X^b)=2 B_{ab}\nabla^a X^b =(B,\cL_X g).
\end{equation}
In particular if $X$ is a conformal vector field then 
$$
\nabla^a(B_{ab}X^b) =\frac{1}{n} V \div X
$$
where $V$ is the metric trace of $B$, i.e.\ $V:=g^{ab}B_{ab}$, and
$\div X =\nabla_a X^a$. 
So 
\begin{equation}\label{prelim}
n\int_{\partial M} B_{ab} X^a \nu^b d\sigma_g= \int_M V (\div X ) dv_g~.
\end{equation} 

A related identity arises from the (metric) trace-free part of $B_{ab}$, 
that is 
$$
\BO_{ab}:= B_{ab}-\frac{1}{n}g_{ab}V.
$$
If $X$ is a conformal Killing vector field then 
$$
\nabla^a(\BO_{ab} X^b)= (\nabla^a\BO_{ab})X^b + \BO_{ab}\nabla^a X^b.
$$
But then $\BO_{ab}\nabla^a X^b=0$, since $\BO_{ab}$ is symmetric
trace-free, while $\frac{1}{2}(\nabla^a X^b+(\nabla^b X^a)= \frac{1}{n}g^{ab}\div
X$. For the other term observe that
$$
X^b\nabla^a\BO_{ab}= X^b(\nabla^aB_{ab}-\frac{1}{n}\nabla_b V)=-\frac{1}{n}X^b\nabla_b V.
$$
Thus 
$$
\begin{array}{rl}
\int_M \cL_X V~ dv_g & =  - n\int_{M} \nabla^a(\BO_{ab} X^b) dv_g\\
&= - n\int_{\partial M} \BO_{ab} X^a \nu^b d\sigma_g
\end{array}
$$

\smallskip

\noindent Recalling also \nn{div}, we summarise as follows.
\begin{theorem}\label{maini} On an oriented  Riemannian manifold $M$ with
  boundary $\partial M$ the following holds.  If $B$ is a locally
  conserved symmetric 2-tensor, of compact support, and $X$ is a
  conformal vector field, then
\begin{equation}\label{b}
\int_M \cL_X V ~ dv_g= - n \int_{\partial M} \BO_{ab} X^a \nu^b d\sigma_g,
\end{equation}
where $V$ is the metric trace of $B$, i.e.\ $V=g^{ab}B_{ab}$.  In
particular this holds for any gradient tensor or generalised energy-momentum tensor $B$
that has compact support.
\end{theorem}

In particular the above applies when $\partial M=\emptyset$. Thus we have 
another  Kazdan-Warner type result. For emphasis we state this specialisation.
\begin{corollary}\label{kwa}
  On a Riemannian manifold $M$, without boundary, let $V$ be the metric trace of
   a compactly supported and locally conserved symmetric 2-tensor $B$.
Then for any conformal vector field $X$ we have 
$$
\int_M \cL_X V dv_g=0.
$$
\end{corollary}

Using Proposition \ref{self} and the result \nn{confv}, Theorem
\ref{maini} also gives the following result.
\begin{corollary} \label{ci} Suppose that the natural scalar invariant
  $V$ is naturally conformally variational on a compact $n$-manifold with
  boundary. Then $V=g^{ab}B_{ab}$ where $B_{ab}$ is a natural gradient of
  some Riemannian functional, and
  the relation \nn{b} holds for any conformal vector field $X$. If $V$
  has a well defined weight $\ell\neq -n$, then $B$ is gradient of the
  functional
$$
\cS (g) =\frac{1}{n+\ell}\int_M V dv_g.
$$
\end{corollary}

\medskip

A special case of Theorem \ref{maini} arises 
when $B_{ab}$ is (a non-zero multiple of) the {\em Einstein tensor}
$$
B_{ab}:=P_{ab}-g_{ab}J,
$$
which is the gradient arising from the Einstein-Hilbert action;
here $n\geq 3$ and we assume compact support.  So then $V=(1-n)J$.
Here $P_{ab}$ is the Schouten tensor and $J=g^{ab}P_{ab}$; in terms of
the Ricci and scalar curvatures, this is characterised by
$$
\Ric_{ab}=(n-2)P_{ab}+ J g_{ab}, 
$$
whence $\Sc = 2(n-1)J$.
So then ($2\times$) \nn{b} states
$$
2(1-n) \int_M \cL_X J ~ dv_g =- 2n\int_{\partial M} P_{(ab)_0}X^a \nu^b d\sigma_g,
$$
where $(\cdots)_0$ indicates the trace-free symmetric part.
In other terms we obtain, 
$$
 \int_M \cL_X \Sc dv_g = \frac{2n}{n-2}\int_{\partial M} \Ric_{(ab)_0}X^a \nu^b d\sigma_g,
$$
as a special case of Theorem \ref{maini}. This is precisely the Schoen
identity \nn{sch} from the introduction.

\begin{remark}
  The identity \nn{prelim} is widely used in the literature, see e.g.\
  \cite{B,BE,BR,ezin,OP} and references therein. 
\end{remark}

\begin{remark}
  Theorem \ref{maini} produces a Schoen-type identity for every
  locally conserved symmetric 2-tensor, and thus in particular for
  every natural gradient, or generalised energy-momentum tensor.  The
  surprising aspect of the Theorem is that it provides a rather subtle
  global relation between the trace and trace-free parts of a locally
  conserved 2-tensor. 
\end{remark}

\subsection{Conserved quantities }\label{consS}
A Killing vector field $X$ is of course also a conformal Killing
vector field.  However Theorem \ref{ukw} and Corollary \ref{kwa} are
vacuous for such $X$: if $X$ is a Killing vector then for any function
$f$ on a closed Riemannian manifold $M$ we have $\int_M \cL_X
f~dv_g=0$, since $\div X=0$ and so $\cL_X f= \div (f X)$.  In both
Theorem \ref{ukw} and Corollary \ref{kwa} the function $V$ is the
trace of a locally conserved symmetric 2-tensor $B$. Thus these
results are also obviously vacuous if in fact $B$ is trace-free, so
$B=\BO$, even if $X$ is not Killing.  

It is natural to ask of the meaning of the corresponding
Pohozaev-Schoen type identities in these degenerate cases. This brings
us to the
following result, at least part of which is well
known in the physics literature (see e.g.\ \cite{HE}).
\begin{proposition}\label{std} Suppose that $X$ is a (conformal) Killing vector
  field on a Riemannian manifold, and $B$ is a locally conserved
  (metric trace-free) symmetric 2-tensor. Then the corresponding
  current $ J_a:=B_{ab}X^b $ is locally conserved, that is
\begin{equation}\label{conse}
\div J=0.
\end{equation}
\end{proposition}
\begin{proof}
  If $B$ is a symmetric 2-tensor that satisfies $\nabla^aB_{ab}=0$
  then for any vector field $X^b$, and setting $ J_a:=B_{ab}X^b $, it
  follows immediately from \nn{key1} that 
  $\nabla_a J^a$ is zero if and only if $\cL_X g$ is pointwise
  orthogonal to $B$. Thus in particular if $\cL_X g=0$ this holds.  It
  also holds if instead $\cL_X g=\frac{2}{n}(\div X) g$, provided $B$
  its trace-free.
\end{proof}
For either of the cases in the Proposition, it is easily seen that
the Pohozaev-Schoen type identity of Theorem \ref{maini} is equivalent to the
usual flux conservation law for conserved currents. 

On the other hand if $B$ is locally conserved but not necessarily
trace-free then from \nn{key1} we have, in our current notation, $\div
J= V \div X /n$, for a conformal Killing vector field $X$. Then on the
left-hand-side of \nn{prelim} $\int_{\partial
  M}B_{ab}X^a\nu^b~d\sigma_g= \int_{\partial M} J_a\nu^a ~d\sigma_g $ is
a measure of flux reflecting conservation failure.

Thus we see that the identity of Theorem \ref{maini} is exactly a
complement of the usual conservation law for conserved currents. To
underscore this point we note here that the Proposition above provides
a route to proliferating conserved quantities on geometries with
symmetry.
\begin{theorem} \label{cons} Each natural scalar invariant $L$
determines
  a corresponding natural gradient
\begin{equation}\label{gradi}
B^L_{ab},
\end{equation}
and so the following:
\begin{itemize}
\item On any Riemannian manifold with a Killing vector field $X$ one
  obtains a corresponding canonical and locally conserved current
  $J^L_a$, (i.e. $J^L$ satisfies \nn{conse}).

\item If $L$ has the property that, on closed manifolds, $\cS(g)=\int_M
  L dv_g$ is conformally invariant, then $B^L_{ab}$ is conformally
  covariant 
and trace-free. It follows that on any Riemannian manifold
  with a conformal Killing vector field $X$ the corresponding
  canonical and locally conserved current $J^L_a$ is conformally
  covariant.  In this case the local conservation equation \nn{conse}
  is conformally invariant.

\item In either case $L$ determines a non-local invariant 
$$
I^L_{\Sigma}:=\int_{\Sigma} J^L_a d\sigma^a,
$$
for each hypersurface $\Sigma$, with the property that
$I^L_{\Sigma_1}=I^L_{\Sigma_2}$ if $\Sigma_1$ and $\Sigma_2$ are
homologous hypersurfaces  sharing the same boundary. If $\cS(g)=\int_M L
dv_g$ is conformally invariant then $I^L_\Sigma$ is conformally
invariant.

\end{itemize}

\end{theorem}
\begin{proof}
  We observed in section \ref{total} that on closed manifolds the
  action determined by $L$, viz.\ $S(g)=\int_M L ~dv_g$ has a
  corresponding natural gradient $B^L_{ab}$, and this is locally
  conserved, cf.\ \nn{div}. Then, as a natural tensor, $B^L_{ab}$ is
  given by a universal formula in terms of partial (metric or volume
  form) contractions of Levi-Civita covariant derivatives of the
  Riemannian curvature. We now take this universal formula as defining
  the symmetric and locally conserved tensor $B^L_{ab}$.

Thus the first result then follows from
  Proposition \ref{std} with
\begin{equation}\label{jl}
J^L_a:= B^L_{ab}X^b.
\end{equation}

Now set $V:=g^{ab}B_{ab}^L$.  If $\cS(g)$ is conformally invariant (on
closed manifolds) then \nn{core} must be zero when, for example
$g^t=e^{2t\om}g$, $\om\in C^\infty (M)$. But in this case $h=2\om g$,
so
$$
\int_M \om V ~dv_g=0.
$$ 
This must hold for arbitrary $\om\in C^\infty (M)$, and so $V=0$,
i.e.\ $B^L$ is metric trace-free. Again this must be also true of the
universal formula for $B^L$.

Thus the claim that $J^L$ (as in \nn{jl} with X now conformal Killing)
is conserved, as stated in the second point of the Theorem, also
follows from Proposition \ref{std}. An easy argument involving second
variations of $\cS(g)$, that mix conformal and total metric
variations, then shows that $B^L$ is necessarily conformally invariant
(see e.g.\ \cite{TomSrni} where also the notion of conformal
invariance, as used here, is discussed).

Finally for the second point, if $L$ has a well-defined weight (and
any $L$ is a sum of such) then $\cS(g)$ conformally invariant and
non-trivial implies this weight is $-n$. Since any natural scalar is a
sum of invariants each of which has a well-defined weights, it follows
that we may assume without loss of generality that $L$ has weight
$-n$. It follows that $B^L$, and hence also $J^L$, has weight
$2-n$ and in fact they are then conformally covariant of weight $2-n$. 
In this case it is well known (and easily verified) that the
equation \nn{conse} is conformally invariant.

The third point is then immediate from the divergence theorem, save
for the comment about conformal invariance. But the latter is an easy
consequence of the weight of $J^L$ and its conformal covariance.

\end{proof}

\begin{remark} If $L$ is a coupled scalar invariant, in the sense of
  Section \ref{se}, then we may replace $B^L$, In Theorem \ref{cons},
  by the corresponding generalised energy-momentum tensor.
\end{remark}

\subsection{Other signatures} For simplicity of exposition in the
above we have restricted to Riemannian signature. In fact all results
above in Section \ref{bc} extend as stated to pseudo-Riemannian
manifolds of any signature with the following restrictions and minor
adjustments: the boundary conormal $\nu_a$ is nowhere null; it is
normalised so that $g_{ab}\nu^a\nu^b=\pm 1$; and it satisfies that at
any point of the boundary $\nu_aX^a$ is positive if $X^a$ is an
outward pointing tangent vector.

The restriction that $\nu_a$ be nowhere null can be removed if
statements are adjusted appropriately. We
leave this to the reader.

\section{Examples and Applications} \label{exs}

Theorems \ref{ukw} and \ref{maini} are already very general. For
example begin with {\em any} natural scalar invariant $L$.
Since natural scalar invariants are easily written down using
Weyl's classical invariant theorem \cite{ABP}, we may readily
proliferate examples. Then
generically the total metric variation of $\cS(g):= \int_M L(g) dv_g$
will yield a corresponding non-trivial Euler-Lagrange tensor $B^L$ via
\nn{core}. Exceptions are those natural scalars $L$ whose integral is
a topological (or smooth structure) invariant, such as the Pfaffian in
even dimensions. In any case of $B^L\neq 0$ Theorem \ref{maini} is
 non-trivial.

 If the integral of $L$ is conformally invariant (so the weight of $L$
 is $-n$) then $B^L$ is trace-free (by Theorem \ref{cons}), and the
 left-hand-side of \nn{b} vanishes in Theorem \ref{maini}; the latter
 nevertheless yielding a non-trivial constraint as discussed in
 Section \ref{consS}.  Otherwise, by \nn{confv} we see that
 $V^L:=\tr^g(B^L)$ is a conformally variational natural scalar and
 Theorem \ref{ukw} and Theorem \ref{maini} apply non-trivially.

Suppose that $V^L$ has a well defined weight $\ell\neq -n$ (as follows
if $L$ does).  Then the map 
$$
L\to (n+\ell)^{-1} V^L
$$
may be regarded as a projection to the {\em conformally variational
  part} of $L$, as follows from Proposition \ref{self}. Ignoring
possible deeper applications, this at least shows that conformally
variational scalar invariants are, in a suitable sense, extremely
common. We discuss some cases below.

Note that if $V$ is a weight $\ell\neq -n$ scalar invariant, then it being
conformally variational immediately implies it has some properties which are
analogous to the scalar curvature. In particular we have the following. Let us
write $B^V_{ab}$ for the gradient of the functional $\cS^V(g):=
(n+\ell)^{-1}\int_M V dv_g$. Then $V=g^{ab}B^V_{ab}$ and, denoting by
$\BO^V$ the trace-free part of $B^V$, we have this observation:
\begin{proposition}\label{ccurv}
  If $\BO^V_{ab}=0$ then $V=$~constant.
\end{proposition}
\begin{proof}
This is an immediate consequence of
$\nabla^a B^V_{ab}=0 $.
\end{proof}
\noindent The point is that in the case of $V$ being the scalar
curvature $\BO^V_{ab}=0$ expresses the Einstein equations. In that setting
the result in the Proposition is often viewed as a consequence of the
Bianchi identities, but we see here that it can be seen to arise from
the fact that the Einstein tensor is locally conserved (and so the
proposition may be extended in an obvious way).

The discussion here is still unnecessarily restrictive. Further
examples arise from more general Riemannian functionals (e.g.\ Example
\ref{poly}), the use of generalised energy-momentum tensors and so forth.
We conclude this section with some special cases.

\subsection{Local conformal invariants} If $V(g)$ is a natural
(scalar) conformal invariant of weight $\ell$, meaning that
$V(e^{2\om}g)=e^{\ell \om} V(g)$ then if $\ell\neq -n$ it is easily
verified that $V(g)$ is naturally conformally variational, with
\nn{selfe} giving a functional. For example if $W$ denotes the
Weyl curvature then $|W|^2$ is a weight $-4$ conformal invariant, and
so is conformally variational in dimensions greater than 4.

\subsection{Q-curvatures} On Riemannian $n$-manifolds, there is an
important class of natural scalar curvature quantities $Q_m$,
parametrised by positive even integers $m$ with
$m\notin\{n,n+2,n+4,\ldots\}$, which are sometimes termed {\em
  subcritical Q-curvatures} \cite{Tomsharp}. 
In a conformal sense these generalise the scalar
curvature: $Q_2$ is the scalar curvature ($n\geq 3$) and if
$\widehat{g}=e^{2\om}g$, $\om\in C^\infty(M)$, then
\begin{equation}\label{Qt}
 Q^{\widehat{g}}_m = u^{\frac{n+m}{m-n}}\left(\delta
  S^g_md + Q^g_m\right)u,
\end{equation}
where, $u=e^{\frac{n-m}{2}\om}$, $\delta=-\div$ is the formal adjoint
of the exterior derivative $d$ and $S^g_m$ is an appropriate operator.
The differential operator $P_m: = \delta S^g_md + Q^g_m $ is
conformally invariant, and is ($\frac{2}{n-m}\times)$ the GJMS
operator \cite{GJMS} with leading term the Laplacian power
$\Delta^{m/2}$. Thus \nn{Qt} is a higher order analogue of the Yamabe
equation (which controls scalar curvature prescription). Considering
now a curve $\widehat{g}=e^{2 s \om}g$, and differentiating at $s=0$, we find
$$
(Q_m)^\bullet= -m Q_m \om + \frac{n-m}{2} \delta S^g_md \om.
$$
It follows easily that $Q_m$ is naturally conformally variational and
arises from an action as given in Proposition \ref{self} (with
$\ell=-m$). Thus on closed manifolds the $Q_m$ are constrained by
\nn{confN} of Theorem \ref{ukw}.

\medskip

The critical Q-curvature $Q_n$ is a weight $-n$ Riemannian invariant
on even $n$-manifolds, and is conformally variational
\cite{Tomsharp,BrO91}, although not known to be naturally so. Thus it
satisfies the Kazdan-Warner type identity of Theorem \ref{bway} (and
cf.\ \cite{delRob} who first prove this and also the subcritical
cases). In dimension 2 the critical $Q$ curvature is the Gauss
curvature and so is also covered by Theorem \ref{nocvt}. It seems
likely that the higher dimensional critical Q-curvatures could
also be treated this way, but we shall not take that up here.  An easy
proof using conformal diffeomorphism invariance also follows from
Theorem 7.1 of \cite{BrGoPont}.

In summary: There are Q-curvatures $Q_m$ for even integers
$m\notin\{n+2,n+4,\ldots\}$ and the following holds.
\begin{proposition}\label{Qcase}
For any conformal vector field $X$ on a closed
  $(M^n,g)$, $n\geq 2$, we have
$
\int_M \cL_X Q_{m} ~ dv_g=0.
$ 
\end{proposition}

In dimensions $n\geq 2$, $Q_2$ is a non-zero multiple of the scalar
curvature. Explicit formulae for the Q-curvatures $Q_4$, $Q_6$, and
$Q_8$, as well as an algorithm for generating the higher $Q_m$, may be
found in \cite{GoPet}. An alternative algorithm may be found in \cite{GrH}.
A recursive approach for the Q-curvature is developed in \cite{JQ}.

\subsection{Higher Einstein tensors}\label{hein}

Throughout the following we work on a manifold of dimension $n\geq 3$
and take $m\in 2\mathbb{Z}_{>0}$, with $m\notin\{n+2,n+4,\ldots\}$.  With
$Q_m$, as above and $\dim M=n\geq 3$, we define a class of natural
tensors.
\begin{definition}\label{heind}
  Let $E^{(m)}$ be the symmetric natural 2-tensor defined by \nn{core}
  (i.e.\ $E^{(m)}:=B$) where
$$
\cS(g):= (n-m)^{-1}\int_M Q_m^g dv_g, \quad \mbox{if}\quad m\neq n ,
$$
and~ $\cS(g):= \int_M Q_m^g dv_g$, if $m=n$.  Then we shall
call $E^{(m)}$ a \underline{higher Einstein tensor}.
\end{definition}
The term ``higher Einstein'' is partly suggested by \nn{Qt} and the following:\\
\begin{itemize}  
\item For $m=2$, and $n\geq 3$, $E^{(m)}$ is the usual Einstein tensor
  (up to a non-zero constant).

\item Since each $E^{(m)}$ arises as a total metric variation we have 
$$
\nabla^a E^{(m)}_{ab}=0,
$$
as a special case of \nn{div}.

\item Proposition \ref{ccurv} holds with $V= Q_m$ and $B^V=E^{(m)}$, 
 with $m\neq n$.

\item $Q_m =g^{ab} E^{(m)}_{ab} $, for  $m\neq n$,
and on Einstein manifolds $Q_m$ is
  constant \cite{FGambnew,Gopowers}.

\end{itemize}

\begin{remark} In the case of even manifolds $M^n$ and $m=n$,
  $E^{(m)}$ is the Fefferman-Graham obstruction tensor of
  \cite{FGast}, see \cite{GrH}.  (In dimension $n=4$ this is the
  well-known Bach tensor.) Thus in this case $E^{(m)}$ is trace-free
  and conformally invariant.
\end{remark}

The following is a special case of Theorem \ref{maini}. 
\begin{proposition}\label{pSE} Let $X$ be a conformal vector field 
on a compact manifold $M$ with boundary $\partial M$. Then 
\begin{equation}\label{Qsch}
\int_N \cL_X Q_m ~ dv_g= - n\int_{\partial N} \EO^{(m)}_{ab} X^a \nu^b d\sigma_g,
\end{equation}
for $m\neq n$, and where $ \EO^{(m)}_{ab}$ is the trace-free part of
$E^{(m)}_{ab}$.
\end{proposition}

\noindent Thus on even manifolds we may view the $E^{(m)}$ as
``interpolating'' between the usual Einstein tensor and the
Fefferman-Graham obstruction tensor. The latter vanishes on Einstein
manifolds \cite{FGast,GoPetOb,GrH}.
These observations suggest an interesting problem:
\begin{quote}
\noindent{\bf Question:} Do the $\EO^{(m)}_{ab}$ vanish on Einstein manifolds?
\end{quote}
Since we posed this it has been observed by Graham that there a simple
argument confirming that the answer is yes. So the higher Einstein
tensors provide a strict weakening of the Einstein condition.

\begin{theorem}\label{HET} If $(M,g)$ Einstein then $\EO^{(m)}_{ab}=0$
for all  $m\in 2\mathbb{Z}_{>0}$, with $m\notin\{n+2,n+4,\ldots\}$.
\end{theorem}
\begin{proof} \cite{Grpriv} On any Riemannian manifold, the
  Q-curvatures may be given by formulae, the terms of which are simply
  complete metric contractions of covariant derivatives of the Ricci
  curvature, see Proposition 3.5 of \cite{FGambnew}, and the subsequent
  discussion there.  On the other hand in (3.20) of the same source it
  is observed that a metric variation $h=d g^t/dt|_{t=0}$ induces a
  variation of the Ricci curvature which may be expressed purely in
  terms of covariant derivatives of $h$. Specifically:
$$
\frac{d}{dt}\Big|_{t=0} \Ric_{ij}(g^t)= \frac{1}{2}(\nabla^k \nabla_j
h_{ik} + \nabla^k \nabla_i h_{jk} - \nabla^k \nabla_k h_{ij} -
\nabla_i \nabla_j h^k{}_{k}).
$$
The induced variation of the Levi-Civita takes a similar form
$$
\frac{1}{2}g^{k\ell}(\nabla_j h_{i\ell}+\nabla_i h_{j\ell}-\nabla_\ell h_{ij}).
$$

Putting these things together it follows easily that, on any
Riemannian manifold, there is a formula for the $E^{(m)}_{ab}$ which
is a linear combination of terms, each of which is a partial metric
contraction of covariant derivatives of the Ricci curvature. Again no
other curvature is involved in the formula. It follows easily that on an 
Einstein manifold $E^{(m)}_{ab}$ is simply a constant multiple of the metric.
\end{proof}

\begin{remark} Note that the constancy of the Q-curvatures on Einstein
  manifolds (mentioned earlier) is seen to be consistent with Theorem
  \ref{HET}, by dint of Proposition \ref{ccurv}, at least for $m\neq
  n$. (Of course to establish the result that the Q-curvatures are
  constant in this setting, including now the critical case, one would
  more easily use the first line of the proof of Theorem
  \ref{HET}. The fact that one can argue this way was pointed out for
  the GJMS operators in the paragraph after Proposition 7.9 of
  \cite{FGambnew}.)

  There is an analogue of Theorem \ref{HET}, and its proof, for the
  gradient (as in \nn{core}) of any natural scalar field arising as
  the restriction of a natural scalar on the Fefferman-Graham ambient
  manifold. In particular this applies to the $\BO^{(k)}_{ab}(g)$
  arising from the renormalised volume coefficients, as discussed in
  Proposition \ref{vkKW} below. These also vanish on Einstein
  manifolds.
\end{remark}
\begin{remark}
  In \cite{Gursky} Gursky makes several interesting remarks concerning
  the gradients $E^{(m)}_{ab}$. These are related to some of the ideas
  of Section \ref{subsch}. Surprisingly he is also able to define an
  analogue of $\EO^{(4)}_{ab}$ for conformally flat 4-manifolds, and
  (in this setting) this yields an identity of the form \nn{Qsch} for
  the critical Q-curvature. It would be interesting to investigate
  whether his tensor can be derived from a symmetry principle.
\end{remark}

\subsection{Renormalised volume coefficients} \label{rvc} Beginning
with a manifold $(M^n,g)$ $n\geq 3$, these natural scalar invariants
$v_{k}$ arise (see e.g. \cite{Grsrni}) in the problem of \cite{FGast}
of finding a 1-parameter $h_r$ of metrics, with $h_0=g$ and so that
$$
g_+:=\frac{dr^2+h_r}{r^2}
$$
is an asymptotic solution to $\Ric^{g_+}=-n g_+$ along $r=0$ in
$M_+:=M\times (0,\epsilon)$. The renormalised volume coefficients
$v_k$ are defined by a volume form expansion
$$
\left( \frac{\det g_\rho}{ \det g_0} \right) \sim 1 +
\sum_{k=1}^\infty v_k\rho^k ,
$$
in the new variable $\rho=-\frac{1}{2}r^2$ with $g_\rho:= h_r$.  In
odd dimension $n$ this determines $v_k$ for $k\in \mathbb{Z}_{\geq
  1}$, but in even dimensions the mentioned formal problem is
obstructed at finite order and so the $v_k$ are in general defined for
$k\in \{1,\cdots ,\frac{n}{2}\}$ (but are defined for $k\in
\mathbb{Z}_{\geq 1}$ in certain special cases, for example if $g$ is
Einstein or locally conformally flat).

Chang and Fang considered the $v_k(g)$ for the Yamabe type problem of
conformally prescribing constant $v_k(g)$ \cite{CF}. They showed that
for $n\neq 2k$ the equation $v_k(g)=$constant is the Euler-Lagrange
equation for the functional $\int_M v_k(g) dv_g$, under conformal
variations satisfying the volume constraint $\int_M dv_g=1$. 
This also follows from \cite[Theorem 1.5]{Gr} where Graham has shown that 
for $k\in \mathbb{Z}$, with $2k\leq n$ if $n$ even,
the infinitesimal conformal variation of the $v_k$ takes the form
\begin{equation}\label{gv}
\frac{d}{dt} v_k(e^{2t\om})|_{t=0} = -2k \om v_k
+\nabla_a(L^{ab}_{(k)}\nabla_b \om),
\end{equation}
with $L^{ab}_{(k)}$  a symmetric tensor (in fact more detail is
given in \cite{Gr}). 
It follows that, for $n\neq 2k$, the $v_{k}$ are naturally conformally
variational. Thus from Theorems \ref{ukw}, \ref{maini} and Corollary
\ref{ci} we have immediately the following.
\begin{proposition}\label{vkKW}
  Let $k\in \mathbb{Z}$, with $2k<n$ if $n$ even. The $v_k$ satisfy
  Theorem \ref{ukw}. Moreover if we write $B^{(k)}_{ab}(g)$ for the
  gradient determined by \nn{core} with $\cS(g):=(n-2k)^{-1}\int_M
  v_k(g)~ dv_g$ then on any compact manifold $N$, of dimension $n\neq
  2k$, with boundary $\partial N$, and with $X$ a conformal vector
  field, we have
$$
\int_N \cL_X v_k ~ dv_g= - n\int_{\partial N} \BO^{(k)}_{ab} X^a \nu^b d\sigma_g,
$$ 
where $ \BO^{(k)}_{ab}$ is the trace-free part of $B^{(k)}_{ab}$.
\end{proposition}

For the $v_k$, with $2k<n$ if $n$ even, this result specialises to
Kazdan-Warner type identities via Corollary \ref{kwa}.
Note that the differential operator on the right-hand-side of \nn{gv}
is formally self-adjoint, so from \cite[Lemma 2(ii)]{BrGovar} (see
Remark \ref{215}) this shows that the $v_k$ are conformally
variational, {\em including} $v_{n/2}$ for $n$ even. Thus from Theorem
\ref{bway}, or equivalently \cite[Theorem 2.1]{delRob}, we extend the
above result as follows.
\begin{proposition}\label{vnKW}
  Let $n=2k$, then for any conformal vector field $X$ on a closed
  $(M^n,g)$ we have
$$
\int_M \cL_X v_{k}~ dv_g=0.
$$ 
\end{proposition}

\begin{remark}
  The Kazdan-Warner type identities for the $v_k$ are first due to
  \cite{GuoHL}. They use \nn{gv} and a specific calculation that
  follows the ideas of \cite{BoE}. Our point is that, since \nn{gv}
  shows that the $v_k$ are conformally variational, the results can
  also be deduced immediately from the general principles. Importantly
  using also the stronger fact that for $k\neq 2n$ the $v_{k}$ are
  naturally conformally variational we also obtain the generalised
  Schoen-type identity of Proposition \ref{vkKW}.

  For $k=1,2$, or when $g$ is locally conformally flat, the $v_k$
  agree with the elementary symmetric functions $\sigma_k(g^{-1}P)$ of
  the Schouten tensor $P$, see \cite{CF,Gr}. So as noted in
  \cite{GuoHL} the Kazdan-Warner type identities for $v_k$ include
  also the similar results of Viaclovsky for the $\sigma_k(g^{-1}P)$,
  \cite{V2}.
\end{remark}

\subsection{Gauss-Bonnet invariants and Einstein-Lovelock Tensors} 
For $k\in \mathbb{Z}_{\geq 1}$ with $k\leq [n/2]$, the
$2k$-Gauss-Bonnet curvature $S^{(2k)}$ is the complete contraction of
the $k^{\rm th}$ tensor power of the Riemann curvature by the
generalised Kronecker tensor, and has the property that in dimension
$2k$ it is exactly the Pfaffian, i.e.\ the Chern-Gauss-Bonnet
integrand (at least up to a nonzero constant). On $(M,g)$ closed and
Riemannian, with $\cS^{(2k)}(g):= 2 \int_M S^{(2k)}(g) ~dv_g$ the
gradient $G^{(2k)}_{ab}=B_{ab}$ (in the sense of \nn{core}) is called
the Einstein-Lovelock tensor \cite{Love25,Labbi} if $2k\neq n$. (If
$2k= n $, then $\cS^{(2k)}(g)$ is a multiple of the
Euler characteristic.) Thus $G^{(2k)}_{ab}$ is locally conserved
$\nabla^a G^{(2k)}_{ab}=0$, for $2k\neq n $ $S^{(2k)}$ is naturally
conformally variational, and as a special case of Theorem \ref{maini}
we have the following.
\begin{proposition}\label{gb}
Let $X$ be a conformal vector field 
on a compact manifold $M$ with boundary $\partial M$. Then for $2k\neq n $
$$
\int_N \cL_XS^{(2k)}  ~ dv_g= 
- \frac{n}{2(n-2k)}\int_{\partial N} \GO^{(2k)}_{ab} X^a \nu^b~ d\sigma_g,
$$ 
where $ \GO^{(2k)}_{ab}$ is the trace-free part of $G^{(2k)}_{ab}$. In
particular on closed manifolds $\int_N \cL_XS^{(2k)} ~ dv_g=0$.
\end{proposition}
\begin{remark}
  The last conclusion giving a Kazdan-Warner type identity is also
  given by \cite{GuoHL} using a direct calculation. Moreover they show
  that this also holds in the case $2k=n$. In fact it is easily
  verified that for $k$,such that the $S^{(2k)}$ are defined, the
  linearisation of the map $\om \mapsto S^{(2k)}(e^{2\om} g_0)$, $\om\in
  C^\infty(M)$, is formally-self-adjoint. So that result may also be
  obtained from Theorem \ref{bway}, or equivalently \cite[Theorem
  2.1]{delRob}. 
\end{remark}

\subsection{Mean curvature of Euclidean hypersurfaces}

Let $(M,g)$ be a codimension one submanifold of Euclidean space $
\mathbb{E}^{n+1}$, with $g_{ab}$ the pullback metric (i.e. the first
fundamental form). We write $\nabla_a$ to denote the Levi-Civita connection
of $g$.  Let us write $II_{ab}$ for the second fundamental form on $M$
induced by the embedding. Then, since $\mathbb{E}^{n+1}$ is flat,
$II_{ab}$ satisfies the contracted Codazzi equation (see e.g.\ \cite{HE}) 
\begin{equation}\label{ceqn}
\nabla^a II_{ab}- n \nabla_b H =0,
\end{equation}
where, as usual, $g^{ab}$ is the inverse to $g_{bc}$ and $H:=
\frac{1}{n} g^{cd}II_{cd} $ is the mean curvature of the embedding.
Thus the symmetric 2-tensor 
$$
B_{ab}:= II_{ab}-n g_{ab} H
$$
is locally conserved everywhere on $M$: $\nabla^a B_{ab}=0$. It
follows immediately that Theorem \ref{maini} gives a Pohozaev-Schoen type
identity on $M$ (which we may take to have a boundary) with $V=
n(1-n)H$. In particular as a special case of Corollary \ref{kwa} we
recover the following result.
\begin{theorem} \label{AHA} Let $\i : S^n\to \mathbb{E}^{n+1}$ be a
  conformal immersion with mean curvature $H$.  Then
  for any conformal vector field $X$ on $S^n$ we have
$$
\int_{S^n} \cL_X H dv_g=0,
$$
where we view $H$ as a function on $S^n$, and $dv_g$ is the pullback by
$\i$ of the first fundamental form measure. 
\end{theorem}
 
This Theorem is first due to Ammann et al.\ \cite{AHA}. There it is
established using the fact that, on $ \mathbb{E}^{n+1} $, the
restriction of a parallel spinor to $M$ satisfies a certain semilinear
variant of the Dirac equation. They show that any spinor satisfying
such an equation satisfies a Pohozaev-Schoen type identity. The argument
above provides a direct alternative argument for the Kazdan-Warner
type result in Theorem \ref{AHA}; in particular it avoids the use of
spinor fields. The Pohozaev-Schoen identity of \cite{AHA} is interesting and
it would be interesting to investigate whether or not it is a special
case of \nn{b}.

\subsection{The Pohozaev identity}\label{pohS}

That the classical Pohozaev identity of \cite{poh65},
$$
 \lambda n \int_M F(u)
 + \frac{2-n}{2} \lambda\int_M f(u)u
 =\frac{1}{2}\int_{\partial M} (x\cdot \nu)(\nabla_\nu u)^2,
 $$
 follows from the identity of Schoen is stated in \cite{schoen}.  We
 have not been able to find the argument written anywhere so, for the
 convenience of the reader, and since it is an idea that generalises, we
 shall give here the derivation.

For any conformal vector field $X$ on a Riemannian $n$-manifold
$(M,g)$ with smooth boundary $\partial M$, the following identity
holds
$$ 
   \int_M \cL_X \Sc ~dv_g
   = \frac{2n}{n-2}\int_{\partial M}\left(\Ric-\frac{1}{n}\Sc\cdot g \right)(X,\nu)d\sigma_g ;
   $$
 here $\nu$ is the outward normal, and $\Ric$ denotes the Ricci
 curvature.

 We start by taking $M \subset \bbR^n$ with metric $g = u^{4/(n-2)}
 g_0$, $g_0$ the Euclidean metric (with $dvol_{n-1}^0$ on $\partial M$
 and $dvol_n^0$ on $M$ resp.). Then with $p = \frac{n+2}{n-2}$ we have
 that
$$\Sc = -\frac{4(n-1)}{n-2} u^{-p} \Delta u$$
($\Delta$ the Euclidean Laplacian) and we take the (Euler) conformal vector field
$X = x_i \frac{\partial}{\partial x_i}$ (summation convention) to get 
$$ \cL_X \Sc = -\frac{4(n-1)}{n-2} \left( x_i(-p)u^{-p-1}\frac{\partial u}{\partial x_i} \Delta u
+ u^{-p} x_i \frac{\partial}{\partial x_i} \Delta u \right)$$
and so with the relevant volumes, noting that $dvol_n = u^{p+1} dvol_n^0$ and
$dvol_{n-1} = u^{2(n-1)/(n-2)} dvol_{n-1}^0$,
$$ \cL_X \Sc dvol_n = -\frac{4(n-1)}{n-2} \left( x_i(-p)\frac{\partial u}{\partial x_i} \Delta u
+ u x_i \frac{\partial}{\partial x_i} \Delta u \right) dvol_n^0$$
(from now on the volume forms will be understood as the Euclidean ones, and
omitted). In a similar way we can find the boundary term, using $u^{p-1} = e^{2f}$
and $$\Ric = (2-n)[\nabla df - df \otimes df] + [\Delta f - (n-2)|df|^2]g_0$$
which means that
$$\Ric_{ij} = (2-n)[\frac{\partial^2f}{\partial x_i \partial x_j} - \frac{\partial f}{ \partial x_i}
\frac{\partial f}{\partial x_j}] + [\frac{\partial^2f}{\partial x_k^2}
- (n-2)\frac{\partial f}{\partial x_k} \frac{\partial f}{ \partial
  x_k}]\delta_{ij}$$ and similarly in terms of $u$ and its
derivatives. 

Now we first use the
relation between the unit normals $\nu = u^{-2/(n-2)} \nu^0$, and then
consider the identity in cases where $u$ is very small along $\partial
M$; finally taking the limiting case that $u = 0$ on $\partial M$.  We find
Schoen's identity then simplifies and determines the following relation:
$$-\frac{4(n-1)}{n-2} \int_M (-p)x_i \frac{\partial u}{\partial x_i} \Delta u + u x_i 
\frac{\partial}{\partial x_i} \Delta u = $$
$$  
\frac{2n}{n-2}\int_{\partial M} \Bigl( (2-n)[-\frac{p-1}{2}\frac{\partial u}{\partial x_i}
\frac{\partial u}{\partial x_j} - (\frac{p-1}{2})^2 \frac{\partial u}{\partial x_i}
\frac{\partial u}{\partial x_j}] + $$ $$[-\frac{p-1}{2} \frac{\partial u}{\partial x_k}
\frac{\partial u}{\partial x_k} - (n-2) (\frac{p-1}{2})^2 \frac{\partial u}{\partial x_k}
\frac{\partial u}{\partial x_k}]\delta_{ij}\Bigr) \nu_i^0 x_j$$
where as before $p-1 = \frac{4}{n-2}$. Using the fact that $\nabla u$ is normal to the
boundary we obtain
$$4(n-1)\frac{n+2}{n-2} \int_M x_i \frac{\partial u}{\partial x_i} \Delta u
-4(n-1) \int_M u x_i \frac{\partial}{\partial x_i} \Delta u
$$ $$= 2n\frac{2n-2}{n-2} \int_{\partial M} u_{\nu}^2 (\nu^0 \cdot x).$$
Here we have used e.g.
$$\frac{\partial u}{\partial x_i} \frac{\partial u}{\partial x_j} \nu_i^0 x_j = u_{\nu}^2 (\nu^0 \cdot x).$$ 
Now the second integral over $M$ may be integrated by parts (and
no boundary term) to get the new integrand
$$ -x_i \frac{\partial u}{\partial x_j} \frac{\partial^2 u}{\partial x_i \partial x_j}
-u \Delta u$$
so we arrive at, after another integration by parts in the first
term above, this time in the $x_i$ variable,    
$$4(n-1)\frac{n+2}{n-2} \int_M x_i \frac{\partial u}{\partial x_i} \Delta u
-4(n-1) \frac{n+2}{2} \int_M |\nabla u|^2
$$ $$= 2\frac{n+2}{n-2} (n-1) \int_{\partial M} u_{\nu}^2 (\nu^0 \cdot x).$$
With (from the assumptions on $u$ in the Pohozaev identity)
$$\int_M |\nabla u|^2 = \lambda \int_M u f(u)$$
$$\int_M x_i \frac{\partial u}{\partial x_i}\Delta u = n \lambda \int_M F(u)$$        
we finally get
$$2 n \lambda \int_M F(u) - (n-2) \lambda \int_M u f(u)
= \int_{\partial M} u_{\nu}^2 (\nu^0 \cdot x)$$
which is the classical identity we wanted.


\begin{thebibliography}{xx}


\bibitem{AHA} B.\ Ammann, E.\ Humbert,
M.O.\  Ahmedou, {\em An obstruction for the mean
  curvature of a conformal immersion $S^n\to\Bbb R^{n+1}$},
  Proc.\ Amer.\ Math.\ Soc., {\bf 135} (2007), 489--493

\bibitem{ABP} M. Atiyah, R.\ Bott, V.K.\ Patodi, {\em On the heat
    equation and the index theorem}, Invent.\ Math.\ {\bf 19} (1973),
  279--330. Errata:  Invent.\ Math.\  {\bf 28}  (1975), 277--280.

\bibitem{B} P.\ Baird, {\em  Stress-energy tensors and the Lichnerowicz
  Laplacian},  J.\ Geom.\ Phys.\  {\bf 58} (2008), 1329--1342.

\bibitem{BE} P.\ Baird, J.\ Eells, {\em  A
  conservation law for harmonic maps},  Geometry Symposium, Utrecht
  1980 (Utrecht, 1980), pp. 1--25, Lecture Notes in Math., 894,
  Springer, Berlin-New York, 1981.  

\bibitem{BR} P.\ Baird, A.\ Ratto, {\em Conservation laws, equivariant
    harmonic maps and harmonic morphisms}, Proc.\ London Math.\ Soc.\
  (3), {\bf 64} (1992), 197--224.

\bibitem{Besse} A.L.\ Besse, Einstein manifolds. Ergebnisse der
  Mathematik und ihrer Grenzgebiete (3), 10, Springer-Verlag, Berlin,
  1987. xii+510 pp.

\bibitem{Bo86} J.-P.\ Bourguignon, {\em Invariants int\'{e}graux
    fonctionnels pour des \'{e}quations aux d\'{e}riv\'{e}es
    partielles d'origine g\'{e}om\'{e}trique}, in Differential
  geometry, Peniscola 1985, 100--108, Lecture Notes in Math.,
  1209, Springer, Berlin, 1986.

\bibitem{BoE}  J.-P.\  Bourguignon,  J.-P.\ Ezin, {\em  Scalar
  curvature functions in a conformal class of metrics and conformal
  transformations},  Trans.\ Amer.\ Math.\ Soc.\  {\bf 301} (1987), 
  723--736.


\bibitem{Tomsharp} T.P.\ Branson, {\em Sharp inequalities, the
    functional determinant, and the complementary series}, Trans.\
  Amer.\ Math.\ Soc.\ {\bf 347} (1995), 3671--3742.

\bibitem{TomSrni} T.\ Branson, {\em Q-curvature and spectral invariants}, 
 Rend.\ Circ.\ Mat.\ Palermo (2) Suppl.  {\bf No. 75}  (2005), 11--55. 

%\bibitem{BGKV} Branson, . .  {\em Heat Kernel asymptotics with mixed
%    boundary conditions} \edz{for variational boundary objects}

\bibitem{BrGovar} T.\ Branson, A.R. Gover, {\em Variational status of
    a class of fully nonlinear curvature prescription problems},
  Calc.\ Var.\ Partial Differential Equations, {\bf 32} (2008), 
  253--262. 

\bibitem{BrGoPont} T.\ Branson, A.R. Gover, {\em Pontrjagin forms and
    invariant objects related to the $Q$-curvature}, Commun.\
  Contemp.\ Math.\ {\bf 9} (2007), 335--358. 

\bibitem{BrO91} T.P.\ Branson, B.\ \O{}rsted, {\em Conformal geometry
    and global invariants}, Differential Geom.\ Appl.\ {\bf 1} (1991),
  279--308.

\bibitem{CF} S.-Y.A. Chang, H.\ Fang, {\em A class of variational
  functionals in conformal geometry},  Int.\ Math.\ Res.\ Not.\ IMRN 2008,
   Art. ID rnn008, 16 pp.

\bibitem{CY} S.-Y.A. Chang, P.C.\ Yang, {\em Prescribing Gaussian
    curvature on $S^2$}, Acta Math.\ {\bf 159} (1987), 215--259.

\bibitem{delRob} P.\ Delanoe, F.\ Robert, {\em On the local Nirenberg
    problem for the $Q$-curvatures}, Pacific J.\ Math., {\bf 231}
  (2007), 293--304.

\bibitem{ezin} J.-P.\ Ezin, {\em  Remarks on an identity by R.\ Schoen
  and others},  J.\ Nigerian Math.\ Soc.\  {\bf 10} (1991), 19--24.

\bibitem{FGast} C.\ Fefferman, C.R.\ Graham, {\em Conformal
    invariants}, in: The mathematical heritage of \'{E}lie Cartan (Lyon,
  1984).  Ast\'{e}risque 1985, Numero Hors Serie, 95--116.

  \bibitem{FGambnew} C.\ Fefferman, C.R.\ Graham, {\em The ambient metric}, 
 arXiv:0710.0919 [math.DG].



\bibitem{Gopowers} A.R.\ Gover, {\em Laplacian operators and
    $Q$-curvature on conformally Einstein manifolds}, Math.\ Ann.\
  {\bf 336} (2006), 311--334.

\bibitem{GoPet} A.R.\ Gover, L.J.\ Peterson, {\em Conformally
    invariant powers of the Laplacian, $Q$-curvature, and tractor
    calculus}, Comm.\ Math.\ Phys.\ {\bf 235} (2003), 339--378.

\bibitem{GoPetOb} A.R.\ Gover, L.J.\ Peterson, {\em The
  ambient obstruction tensor and the conformal deformation complex},
  Pacific J.\ Math.\  {\bf 226} (2006), 309–351.

\bibitem{Grsrni} C.R.\ Graham, {\em Volume and area renormalizations
    for conformally compact Einstein metrics}, 
%in The Proceedings of the
%  19th Winter School ``Geometry and Physics'' (Srní, 1999).
Rend. Circ. Mat. Palermo (2) Suppl.  {\bf No. 63} (2000), 31--42.

\bibitem{Gr} C.R. Graham, {\em Extended obstruction tensors and
    renormalized volume coefficients}, Adv.\ Math.\ 220 (2009),
  1956--1985.

\bibitem{Grpriv} C.R. Graham, Private communication, November 2010. 

\bibitem{GrH} C.R. Graham, K.\ Hirachi, {\em The ambient obstruction
    tensor and $Q$-curvature}, in AdS/CFT correspondence: Einstein
  metrics and their conformal boundaries, 59--71, IRMA
  Lect.\ Math.\ Theor.\ Phys., 8, Eur.\ Math.\ Soc., Z\"{u}rich, 2005.

\bibitem{GJMS} C.R. Graham, R.\ Jenne, L.J.\ Mason, G.A.J.\ Sparling, 
{\em Conformally invariant powers of the Laplacian. I. Existence},
J.\ London Math.\ Soc.\ (2) {\bf 46} (1992), 557--565. 

\bibitem{GuoHL} B.\ Guo, Z.-C.\ Han, H.\ Li, {\em Two Kazdan-Warner type
    identities for the renormalized volume coefficients and the
    Gauss-Bonnet curvatures of a Riemannian metric},
    arXiv:0911.4649 [math.DG]

\bibitem{Gursky} M.J.\ Gursky, {\em Uniqueness of the functional
  determinant},  Comm.\ Math.\ Phys.\  {\bf 189} (1997), 655--665.

\bibitem{HE} S.W.\ Hawking, G.F.R.\ Ellis, The large scale structure
  of space-time,  Cambridge Monographs on Mathematical Physics,
  No. 1. Cambridge University Press, London-New York, 1973. xi+391 pp.

\bibitem{Hil} D.\ Hilbert, {\em Die grundlagen der physik}, Nach. Ges. Wiss.,
  G\"{o}ttingen, (1915), 461--472.

\bibitem{JQ} A.\ Juhl, {\em On the recursive structure of Branson's
    Q-curvature}, arXiv:1004.1784 [math.DG]

\bibitem{KW74} J.L. Kazdan, F.W.\ Warner, {\em Curvature functions for
  compact $2$-manifolds},  Ann.\ of Math.\ (2) {\bf 99} (1974), 14--47.

\bibitem{KW75} J.L. Kazdan, F.W.\ Warner, {\em Scalar curvature and
  conformal deformation of Riemannian structure},  J.\ Differential
  Geometry {\bf 10} (1975), 113--134.

  \bibitem{Labbi} M.L. Labbi, {\em Variational properties of the
    Gauss-Bonnet curvatures}, Calc.\ Var.\ Partial Differential
  Equations, {\bf 32} (2008), 175--189.

  \bibitem{L-F} J.\ Lelong-Ferrand, {\em Transformations conformes et
      quasiconformes des vari\'{e}t\'{e}s riemanniennes; application
      \`{a} la d\'{e}monstration d'une conjecture de A. Lichnerowicz},
    C. R. Acad. Sci. Paris S\'{e}r, {\bf A-B 269} (1969), A583--A586.

\bibitem{Love25} D.\ Lovelock, {\em The Einstein tensor and its
    generalizations}, J.\ Mathematical Phys.\ {\bf 12} (1971),
  498--501.

\bibitem{O} M.\ Obata, {\em The conjectures on conformal
    transformations of Riemannian manifolds}, J.\ Differential
  Geometry {\bf 6} (1971/72), 247--258.


\bibitem{OPS} B.\ Osgood, R.\ Phillips, P.\ Sarnak, {\em Extremals of
    determinants of Laplacians}, J.\ Funct.\ Anal., {\bf 80} (1988),
  148--211.

\bibitem{OP} A.\ Pierzchalski, B.\ \O{}rsted, 
{\em The Ahlfors Laplacian on a Riemannian manifold with boundary},  
Michigan Math.\ J.\  {\bf 43}  (1996), 99--122.

\bibitem{PRS} S.\ Pigola,  M.\ Rigoli, 
A.G.\  Setti, {\em Some applications of integral formulas
  in Riemannian geometry and PDE's}, Milan
  J.\ Math.\ {\bf 71} (2003), 219--281


\bibitem{poh65} S.\ Pohozaev, {\em On the eigenfunctions of the
    equation $\Delta u+\lambda f(u)=0$}, Dokl.\ Akad.\ Nauk
  SSSR {\bf 165} 1965 36--39. (Translation: Soviet Math. Dokl. {\bf 6}
  (1965), 1408--1411]).

\bibitem{P} A.M.\ Polyakov,  {\em Quantum geometry of fermionic strings},
Phys.\ Lett.\ B {\bf 103} (1981), 211--213. 

\bibitem{PS} P.\ Pucci, J.\ Serrin, {\em A general variational
  identity},  Indiana Univ.\ Math.\ J.\  {\bf 35} (1986), 681--703.


\bibitem{schoen} R.M.\ Schoen, {\em The existence of weak solutions with
    prescribed singular behavior for a conformally invariant scalar
    equation},  Comm. Pure Appl. Math. 41 (1988), no. 3, 317--392.

\bibitem{Stredder} P.\ Stredder, {\em Natural differential operators on
    Riemannian manifolds and representations of the orthogonal and
    special orthogonal groups},  J.\ Differential Geometry {\bf 10}
    (1975), 647--660.


\bibitem{V2} J.\ Viaclovsky, {\em Some fully nonlinear equations in
  conformal geometry}, in  Differential equations and mathematical physics
  (Birmingham, AL, 1999), 425--433, AMS/IP Stud. Adv. Math., 16,
  Amer. Math. Soc., Providence, RI, 2000.


\bibitem{Wag} A.\ Wagner, {\em Pohozaev’s Identity from a Variational
    Viewpoint}, J.\ Math.\ Anal.\ Appl., {\bf 266} (2002), 149--159.

\end{thebibliography}
\end{document}